\documentclass[12pt,a4paper,reqno]{amsart}
\usepackage{amssymb}
\usepackage{amscd}
\numberwithin{equation}{section}

\theoremstyle{plain}

\newtheorem{theorem}[subsection]{Theorem}
\newtheorem{proposition}[subsection]{Proposition}
\newtheorem{lemma}[subsection]{Lemma}
\newtheorem{corollary}[subsection]{Corollary}
\newtheorem{conjecture}[subsection]{Conjecture}

\theoremstyle{definition}

\newtheorem{definition}[subsection]{Definition}
\newtheorem{remark}[subsection]{Remark}

\newtheorem{example}[subsection]{Example}

\newcommand\eps{\varepsilon}
\newcommand\R{{\mathbf{R}}}
\renewcommand\P{{\mathbb{P}}}
\newcommand{\Z}{{\mathbf{Z}}}
\newcommand\E{{\mathbb{E}}}
\newcommand\A{{\mathbb{A}}}
\newcommand\ra{\rightarrow}
\newcommand{\beq}{\begin{equation*}}
\newcommand{\eeq}{\end{equation*}}
\newcommand{\Aut}{\mbox{Aut}}
\newcommand{\Gr}{\operatorname{Gr}}

\newcommand{\ic}{\mathfrak}
\newcommand{\PP}{\mathcal{P}}

\begin{document}

\title[Kakeya for algebraic varieties over finite fields]{The Kakeya set and maximal conjectures for algebraic varieties over finite fields}

\author{Jordan S. Ellenberg}
\address{Department of Mathematics, University of Wisconsin, Madison WI 53706}
\email{ellenber@math.wisc.edu}

\author{Richard Oberlin}
\address{UCLA Department of Mathematics, Los Angeles, CA 90095-1555}
\email{oberlin@math.ucla.edu}

\author{Terence Tao}
\address{UCLA Department of Mathematics, Los Angeles, CA 90095-1555}
\email{tao@math.ucla.edu}


\begin{abstract}  Using the polynomial method of Dvir \cite{dvir}, we establish optimal estimates for Kakeya sets and Kakeya maximal functions associated to algebraic varieties $W$ over finite fields $F$.  For instance, given an $n-1$-dimensional projective variety $W \subset \P^n(F)$, we establish the Kakeya maximal estimate
$$
\| \sup_{\gamma \ni w} \sum_{v \in \gamma(F)} |f(v)| \|_{\ell^n(W)} \leq C_{n,W,d} |F|^{(n-1)/n} \|f\|_{\ell^n(F^n)}$$
for all functions $f: F^n \to \R$ and $d \geq 1$, where for each $w \in W$, the supremum is over all irreducible algebraic curves in $F^n$ of degree at most $d$ that pass through $w$ but do not lie in $W$, and with $C_{n,W,d}$ depending only on $n, d$ and the degree of $W$; the special case when $W$ is the hyperplane at infinity in particular establishes the Kakeya maximal function conjecture in finite fields, which in turn strengthens the results of Dvir.
\end{abstract}

\maketitle
\today

\section{Introduction}

Recently, Dvir\cite{dvir} established the following result, first conjectured by Wolff\cite{wolff:survey}:

\begin{theorem}[Kakeya set conjecture for $F^n$]\label{dvir-thm}  Let $n \geq 1$, let $F$ be a finite field, and let $E \subset F^n$ be a subset of $F^n$ that contains a line in every direction.  Then $|E| \geq c_n |F|^n$, where $c_n > 0$ depends only on $n$.
\end{theorem}

\begin{remark} Dvir's original argument gave the value $c_n = 1/n!$ for the constant $c_n$; this was recently improved in \cite{saraf}, \cite{dvir2} to $c_n = (\frac{1}{2} + o(1))^n$, which is best possible except for possible refinement of the $o(1)$ error.
\end{remark}

This result is the finite field analogue of the \emph{Kakeya set conjecture for $\R^n$}, which is of importance in harmonic analysis; we refer the reader to \cite{wolff:survey}, \cite{gerd:kakeya} for further discussion.  In this paper we shall obtain a stronger version of Theorem \ref{dvir-thm}, the finite field analogue of an analogous conjecture in $\R^n$ (see e.g. \cite{borg:kakeya}, \cite{wolff:survey}.)  This answers a question raised in \cite{gerd:kakeya}.

\begin{theorem}[Kakeya maximal conjecture for $F^n$]\label{max-thm}  Let $n \geq 1$, let $F$ be a finite field, and let $f: F^n \to \R$ be a function.  Then
\begin{equation}\label{exp}
\| \sup_{\gamma // \omega} \sum_{x \in \gamma} |f(x)| \|_{\ell^n(\P^{n-1}(F))} \leq C_n |F|^{(n-1)/n} \|f\|_{\ell^n(F^n)},
\end{equation}
\end{theorem}

In the statement of Theorem~\ref{max-thm} and throughout the paper, we use the following notation:
\begin{itemize}
\item $\P^{n-1}(F)$ is the $n-1$-dimensional projective space over $F$, and a point $\omega$ of $\P^{n-1}(F)$ is interpreted as a direction in $F^n$;
\item $\ell^p$ denotes the usual family of Lebesgue norms
\beq
\| G \|_{\ell^p(A)} := (\sum_{\omega \in {A}} |G(\omega)|^p)^{1/p}
\eeq
with the usual modification for $p=\infty$.
\end{itemize}

The supremum on the left-hand side of \eqref{exp} is over all lines $\gamma$ with direction $\omega$, and $C_n$ is a constant depending only on $n$.  In general, the bounds in this paper will all involve unspecified constants depending on geometric quantities like $n$ or the degree and dimension of various fixed algebraic varieties; the point is that these constants are independent of $|F|$.

\begin{remark} By interpolating \eqref{exp} with more trivial estimates we obtain
\begin{equation}\label{shoop}
\| \sup_{\gamma // \omega} \sum_{x \in \gamma} |f(x)| \|_{\ell^q(\P^{n-1}(F))} \leq C_n |F|^{(n-1)/q} \|f\|_{\ell^p(F^n)}
\end{equation}
whenever $p$ and $q$ are positive real numbers satisfying $1 \leq p \leq n$ and $1 \leq q \leq (n-1) \frac{p}{p-1}$.  Various special cases of \eqref{shoop} had been established previously in the literature; see \cite{gerd:kakeya}.  Except for the issue of determining the constant $C_n$, the various exponents in \eqref{exp} or \eqref{shoop} are sharp, as can be seen by testing \eqref{exp} with $f$ equal to the indicator of a point, a line, or the whole space $F^n$. 
\end{remark}

\begin{remark}  It would be interesting to investigate the question of which functions $f$ make the inequality in Theorem~\ref{max-thm} a near-equality.  The indicator function of a linear subspace always yields an equality in \eqref{exp}; when $n$ is large, does approximate equality in \eqref{exp} indicate that $f$ is disproportionately concentrated on a linear subspace of $F^n$?  On the other hand, the dual estimate
$$ \| \sum_{\omega \in \P^{n-1}(F)} g(\omega) 1_{\gamma_\omega} \|_{\ell^{n/(n-1)}(F^n)} \leq C_n |F|^{(n-1)/n} \| g \|_{\ell^{n/(n-1)}(\P^{n-1}(F))},$$
where $\gamma_\omega$ is a line with direction $\omega$ for each $\omega \in \P^{n-1}(F)$, is essentially an equality for any non-negative $g$, thanks to the elementary inequality
$$ \| \sum_{\omega \in \Omega} f_\omega \|_{\ell^p(A)} \geq (\sum_{\omega \in \Omega} \| f_\omega \|_{\ell^p(A)}^p)^{1/p}$$
for all $1 \leq p < \infty$ and non-negative functions $f_\omega: A \to \R^+$.  (We will not use this dual formulation elsewhere in this paper.)
\end{remark}

Theorem \ref{dvir-thm} follows immediately from Theorem \ref{max-thm} by specializing $f$ to equal the indicator function $f=1_E$ of the set $E$ in Theorem \ref{dvir-thm}.   

In fact, the goal of this paper is to establish a still more general statement than Theorem~\ref{max-thm}, in which $\P^{n-1}(F)$ is replaced by the $F$-rational points of an essentially arbitrary $n-1$-dimensional algebraic variety of bounded degree.  (See Appendix \ref{alg} for a review of the relevant terminology from algebraic geometry).

\begin{theorem}[Kakeya maximal conjecture for algebraic varieties over $F$]\label{alg-thm}  Let $N, d, n \geq 1$, let $F$ be a finite field, let $V = \P^N$ (resp. $V = \A^N$), and let $W  \subset V$ be a projective (resp. affine) variety of dimension $n-1$ and degree at most $d$.  Then for any $f: V(F) \to \R$ we have
\begin{equation}\label{kakeq}
\| \sup_{\gamma \ni w} \sum_{v \in \gamma(F)} |f(v)| \|_{\ell^n(W(F))} \leq C_{N,d,n} |F|^{(n-1)/n} \|f\|_{\ell^n(V(F))}.
\end{equation}
Here, $w$ ranges over $W(F)$, and for each $w \in W(F)$, the range of the supremum is the set of all irreducible projective (resp. affine) algebraic curves $\gamma$ in $V$ of degree at most $d$, which contain $w$ but are not contained in $W$. The constant $C_{N,d,n} > 0$ depends only on $N,d,n$ (and in particular is independent of $F$).
\end{theorem}

\begin{remark} The observation that the polynomial method can be used to control the intersection of a specified set with curves of bounded degree (as opposed to lines) already appears in \cite{dvir:dvirwigderson}, in Claim 3.3 and the discussion preceding it.
\end{remark}

\begin{remark} \label{rem:ndg} The non-degeneracy requirement that $\gamma$ is not contained in $W$ is easily seen to be necessary, for instance by considering the case when $V = \A^n$, $W = \A^{n-1}$, and $f$ is the indicator function of $F^{n-1}$. 
\end{remark}

\begin{remark} Theorem \ref{alg-thm} is stated for $V$ equal to all of projective or affine space, but one can trivially restrict $V$ to any smaller set (e.g. some intermediate subvariety between $W$ and $\P^N$ or $\A^N$) and obtain the same estimate for $f \in \ell^n(V(F))$.  Indeed, the most interesting case in practice is that where $V$ is an $n$-dimensional variety containing $W$.  
\end{remark}

Theorem \ref{max-thm} follows immediately from Theorem \ref{alg-thm} by taking $V = \P^{n}$ and $W$ the hyperplane at infinity, and setting $d=1$; then the set of projective lines through a given point of $W(F)$ is naturally identified with the set of lines in $F^n$ parallel to some fixed line.

As a corollary of Theorem \ref{alg-thm}, we have a generalization of Theorem \ref{dvir-thm} to algebraic varieties:

\begin{corollary}[Kakeya set conjecture for algebraic varieties over $F$]\label{setcor} Let $N,d,n,F,V,W$ be as in Theorem \ref{alg-thm}.  Let $w_1,\ldots,w_J \in W(F)$ be distinct $F$-points in $W$, and for each $1 \leq j \leq J$ let $\gamma_j$ be an irreducible algebraic curve of degree at most $d$ in $V$ that passes through $w_j$ but does not lie in $W$.  Let $\lambda \geq 1$, and let $E \subset V(F)$ be a set of $F$-points in $V$ such that $|E \cap \gamma_j(F)| \geq \lambda$ for all $1 \leq j \leq J$.  Then $|E| \geq C_{N,d,n}^{-n} J \lambda^n / |F|^{n-1}$, where $C_{N,d,n}$ is the constant in Theorem \ref{alg-thm}.
\end{corollary}

In Sections \ref{flat-sec} and \ref{general-sec} we shall prove Theorem \ref{alg-thm}.  The strategy will be to first establish the ``flat'' case $V = \A^n$, $W = \A^{n-1}$ of the theorem using (a weighted form of) Dvir's method combined with the Nikishin-Maurey-Pisier-Stein factorization trick and some real interpolation techniques.  Then, using a random projection trick to flatten $W$, we deduce the general case from the flat case.  In Section \ref{variants-sec} we discuss some variants and generalizations of the above results.

\subsection{Notation}

We use the notation $X \ll Y$, $Y \gg X$, or $X = O(Y)$ to denote an estimate of the form $X \leq CY$ where $C$ is a constant that is allowed to depend on parameters such as the dimension $n$ (but is always independent of the field $F$).  We sometimes use subscripts to draw attention to additional parameters on which the implied constants depend; for example, $X \ll_{C_0} Y$ means that $X \leq CY$ where $C$ depends not only on $n$ but also some constant $C_0$ which is local to the argument.   We write $X \asymp Y$ for $X \ll Y \ll X$.

Unless otherwise specified, all varieties are defined over the fixed finite field $F$; so that when we refer to ``a curve of degree $d$ in $\A^n$'', for instance, we mean a curve defined over $F$.


\subsection{Acknowledgements}

The first author is partially supported by NSF-CAREER grant DMS-0448750 and a Sloan Research Fellowship.  The third author is supported by a grant from the MacArthur Foundation, and by NSF grant DMS-0649473.  We are indebted to the anonymous referee for useful comments and corrections.

\section{The flat case}\label{flat-sec}

To prove Theorem \ref{alg-thm}, we first establish a model case when $V=\A^n$ (thus $V(F)=F^n$) and $W$ is the standard hyperplane $W = \A^{n-1}$ of $\A^n$ embedded in $\A^n$ in the usual manner; the reason for treating this case first is that it enjoys a very useful translation-invariance symmetry that can be exploited to obtain good estimates.  In the next section we shall use a random projection trick to reduce the general case to this case.

The main result of this section is

\begin{theorem}[Kakeya maximal conjecture for curves in $F^n$]\label{alg-thm2}  Let $d, n \geq 1$.  For any $f: F^n \to \R$, let $f^*: F^{n-1} \to \R$ be the maximal function
\begin{equation}\label{fw}
 f^*(w) := \sup_{\gamma \ni w} \sum_{v \in \gamma(F) \backslash F^{n-1}} |f(v)|
 \end{equation}
where the supremum is over all irreducible algebraic curves $\gamma$ in $F^n$ of degree at most $d$ that pass through $w$.  Then we have
$$
\| f^* \|_{\ell^n(F^{n-1})} \ll |F|^{(n-1)/n} \|f\|_{\ell^n(F^n)},
$$
where the implied constants can depend on $d$ and $n$.
\end{theorem}

\begin{remark} It is easy to see that the $d=1$ case of Theorem \ref{alg-thm2} implies Theorem \ref{max-thm}, by applying a projective transformation to move $F^{n-1}$ to the plane at infinity (cf. \cite{tao:boch-rest}).  (Strictly speaking, this transformation will leave out a codimension one set of directions, and also a codimension one set of points, but one can easily erase this omission by rotating the estimate a bounded number of times and using the triangle inequality; we omit the details.  In any case, Theorem \ref{max-thm} will also be deduced from the stronger Theorem \ref{alg-thm}, proven in the next section.)
\end{remark}

Theorem \ref{alg-thm2} will follow from the following distributional estimate:

\begin{proposition}[Kakeya distributional estimate for curves in $F^n$]\label{alg-prop}  Let $d, n \geq 1$.  Then there exists $K = K_{d,n}$ such that for any $A > 0$, any $f: F^n \to \{0\} \cup [A,+\infty)$, and every $K \|f\|_{\ell^n(F^n)} \leq \lambda \leq A|F|$, we have
$$
|\{ w \in F^{n-1}: f^*(w) \geq \lambda \}| \ll \frac{|F|^{n-2}}{A \lambda^{n-1}} \|f\|_{\ell^n(F^n)}^n,$$
where the supremum and implied constants are as in Theorem \ref{alg-thm2}.
\end{proposition}

\begin{proof}[Proof of Theorem \ref{alg-thm2} assuming Proposition \ref{alg-prop}]  We will use a variant of the real interpolation method.  We may of course take $f$ to be non-negative and not identically zero.  
We normalize $\|f\|_{\ell^n(F^n)} = 1$.  Using the identity
\begin{equation}\label{distrib}
 \|f^*\|_{\ell^n(F^{n-1})}^n = n \int_0^\infty |\{ w \in F^{n-1}: f^*(w) \geq \alpha \}| \alpha^{n-1}\ d\alpha
\end{equation}
it suffices to show that
$$ \int_0^\infty |\{ w \in F^{n-1}: f^*(w) \geq \alpha \}| \alpha^{n-1}\ d\alpha \ll |F|^{n-1}.$$
The crude bound $|\{ w \in F^{n-1}: f^*(w) \geq \alpha \}| \leq |F|^{n-1}$ allows one to dispose of the region $\alpha \leq C_0$, where $C_0 = O(1)$ is a large constant to be chosen later.  So it remains to show that
\begin{equation}\label{mook}
 \int_{C_0}^\infty |\{ w \in F^{n-1}: f^*(w) \geq \alpha \}| \alpha^{n-1}\ d\alpha \ll |F|^{n-1}.
\end{equation}
Fix $\alpha > C_0$.  For each positive integer $j$ we define a function $f_{j,\alpha}: F^n \ra \R$ by 
\beq
f_{j,\alpha}(v) = \left\{ \begin{array}{ll} f(v) &  \mbox {if}\ 100^{n(j-1)} \frac{\alpha}{C_0^{1/2} |F|} \leq f(v) < 100^{nj} \frac{\alpha}{C_0^{1/2} |F|} \\
0 & \mbox{otherwise} \\
\end{array} \right.
\eeq
We similarly define $f_{0,\alpha}(v)$ to be equal to $f(v)$ when $f(v) < \frac{\alpha}{C_0^{1/2} |F|}$ and $0$ otherwise. Finally, we define $f_\alpha(v)$ to be equal to $f(v)$ when $f(v) \geq 100^{n(j_\alpha-1)} \frac{\alpha}{C_0^{1/2} |F|}$ and $0$ otherwise, where $j_\alpha$ is the largest integer satisfying $\alpha/2^{j_\alpha+1} \geq K$, where $K$ is the quantity in Proposition \ref{alg-prop}.
Then $f$ decomposes as 
\beq
f = f_\alpha + f_{0,\alpha} + \sum_{j=1}^{j_\alpha - 1} f_{j,\alpha}.
\eeq

From Lemma \ref{size-lem} and the bound $f_{0,\alpha}(v) < \frac{\alpha}{C_0^{1/2} |F|}$ we have
$$ f^*_{0,\alpha}(w) \ll \frac{\alpha}{C_0^{1/2}}$$
and thus (if $C_0$ is large enough)
$$ f^*(w) \leq f_\alpha^*(w) + \sum_{j=1}^{j_\alpha -1} f^*_{j,\alpha}(w) + \frac{\alpha}{2}$$
and hence
$$ |\{ w \in F^{n-1}: f^*(w) \geq \alpha\}| \leq |\{ w \in F^{n-1}: f^*_{\alpha}(w) \geq K\}| + \sum_{j=1}^{j_\alpha - 1} |\{ w \in F^{n-1}: f^*_{j,\alpha}(w) \geq \alpha/2^{j+1}\}|.$$
On the other hand, we have
$$ \max(\|f_{\alpha} \|_{\ell^n(F^n)},\|f_{j,\alpha} \|_{\ell^n(F^n)}) \leq  \|f \|_{\ell^n(F^n)} = 1.$$

Applying Proposition \ref{alg-prop} with $A := 100^{n(j_\alpha-1)} \frac{\alpha}{C_0^{1/2} |F|}$ and $\lambda := K$ we obtain
\[
|\{ w \in F^{n-1}: f^*_{\alpha}(w) \geq K\}| \ll_{C_0} 100^{-nj_\alpha} |F|^{n-1} \ll \alpha^{-2n} |F|^{n-1}
\]
where the last inequality follows from the maximality of $j_\alpha.$ It follows that
\[
\int_{C_0}^{\infty} |\{ w \in F^{n-1}: f^*_{\alpha}(w) \geq K\}| \alpha^{n-1}\ d\alpha \ll_{C_{0}} |F|^{n-1} 
\]

For each $j < j_{\alpha}$, we now apply Proposition \ref{alg-prop} with $A := 100^{n(j-1)} \frac{\alpha}{C_0^{1/2} |F|}$ and $\lambda := \min( \alpha/2^{j+1}, A|F|)$; note by the choice of $j_\alpha$ and the hypothesis $\alpha \geq C_0$, that $\lambda \geq K$ if $C_0$ is large enough).  Using the fact that
\beq
A|F| \gg \frac{\alpha}{2^{j+1} C_0^{1/2}}
\eeq
we obtain
$$ |\{ w \in F^{n-1}: f^*_{j,\alpha}(w) \geq \alpha/2^{j+1}\}| \ll_{C_0} 100^{-nj} 2^{(n-1)j} \frac{|F|^{n-1}}{\alpha^n} \| f_{j,\alpha} \|_{\ell^n(F^n)}^n$$
where the subscript $\ll_{C_0}$ indicates that the implied constant can depend on $C_0$.  Putting this all together, one finds that
\begin{multline*} 
\int_{C_0}^{\infty} |\{ w \in F^{n-1}: f^*(w) \geq \alpha\}|\alpha^{n-1}\ d\alpha \\ 
\ll_{C_0} |F|^{n-1} + \sum_{j=1}^\infty 100^{-nj} 2^{(n-1)j} |F|^{n-1} \int_{C_0}^\infty \sum_{v \in \Sigma_{j,\alpha}} f(v)^n\ \frac{d\alpha}{\alpha}
\end{multline*}
 where $\Sigma_{j,\alpha} \subset F^n$ denotes the support of $f_{j,\alpha}$.
 
Now interchange the $\alpha$ integration and $v$ summation; for each fixed $v$, the function of $\alpha$ to be integrated against $d \alpha / \alpha$ is the product of $f(v)^n$ with the characteristic function of an interval of the form $[A, 100^n A]$.  The size of this integral is a constant multiple of $f(v)^n$; so we arrive at the inequality
$$ \int_{C_0}^{\infty}|\{ w \in F^{n-1}: f^*(w) \geq \alpha\}|\alpha^{n-1}\ d\alpha \ll_{C_0} \sum_{j=1}^\infty 100^{-nj} 2^{(n-1)j} |F|^{n-1} \sum_{v \in F^n} f(v)^n$$
and \eqref{mook} now follows from the normalization $\|f\|_{\ell^n(F^n)} = 1$.
\end{proof}

It remains to prove Proposition \ref{alg-prop}.  We first rewrite that proposition in an equivalent form.    We may first normalize by dividing $f$ and $\lambda$ by $A$; thus, it suffices to prove Proposition~\ref{alg-prop} with $A=1$.  Let us enumerate the set $\{ w \in F^{n-1}: f^*(w) \geq \lambda \}$ as $w_1,\ldots,w_J$, and for each $1 \leq j \leq J$, let $\gamma_j$ be the curve attaining the supremum used to define $f^*(w_j)$ in \eqref{fw}.  (Note that as there are only finitely many possible choices for $\gamma_j$, the supremum here is attainable.)  Our task is now to show

\begin{proposition}[Distributional estimate, again]\label{dist}  Let $d, n \geq 1$, and let $w_1,\ldots,w_J$ be distinct points in $F^{n-1}$.  For each $1 \leq j \leq J$, let $\gamma_j$ be an irreducible algebraic curve in $\A^n$ of degree at most $d$ that passes through $w_j$ but does not lie in $\A^{n-1}$.  Let $f: F^n \to \{0\} \cup [1,+\infty)$ and $K \|f\|_{\ell^n(F^n)} \leq \lambda \leq |F|$ for some sufficiently large $K$ depending only on $d,n$, and suppose that
\begin{equation}\label{distal}
\sum_{v \in \gamma_j(F) \backslash F^{n-1}} f(v) \geq \lambda
\end{equation}
for all $1 \leq j \leq J$.  Then
\begin{equation}\label{mong}
J \ll \frac{|F|^{n-2}}{\lambda^{n-1}} \sum_{v \in F^n} f(v)^n.
\end{equation}
\end{proposition}

The hypothesis and conclusion of this proposition are invariant under translations of $F^n$ by vectors in $F^{n-1}$.  We take advantage of this translation invariance to make a standard reduction (cf. \cite{borg:kakeya}, \cite{pisier}, \cite{stein}) to the case of small $\lambda$ (or equivalently, for large $J$):

\begin{proposition}[Nikishin-Maurey-Pisier-Stein factorization reduction]\label{red}  To prove Proposition \ref{dist}, it suffices to do so in the special case $\lambda = K_0 \|f\|_{\ell^n(F^n)} \leq |F|$ for some sufficiently large $K_0$ depending on $n$ and $d$.
\end{proposition}

\begin{proof}  Let $d,n,w_1,\ldots,w_J,\gamma_1,\ldots,\gamma_J,f,\lambda$ be as in Proposition \ref{dist}.  Let $K_0$ be as in Proposition \ref{red}, and assume $K$ sufficiently large depending on $d,n,K_0$.  Let $M$ be the greatest integer less than or equal to $\frac{\lambda^n}{K_0^n \|f\|_{\ell^n(F^n)}^n}$; thus $M \geq 1$ if $K$ is large enough.  

We use a probabilistic method. Let $u_1,\ldots,u_M \in F^{n-1}$ be selected independently and uniformly at random, and let $\Omega := \{ w_j + u_m: 1 \leq j \leq J, 1 \leq m \leq M \}$, thus $\Omega$ is a random subset of $F^{n-1}$.  Observe that every $w \in F^{n-1}$ lies in $\Omega$ with probability
$$ 1 - (1-\frac{J}{|F|^{n-1}})^M \asymp \min( \frac{MJ}{|F|^{n-1}}, 1 ) $$
The expected size of $\Omega$ is thus
$$ \E |\Omega| \asymp \min( MJ, |F|^{n-1} ),$$
and so we may select $u_1,\ldots,u_M$ such that
\begin{equation}\label{om}
|\Omega| \gg \min( MJ, |F|^{n-1} ).
\end{equation}
Fix these $u_1,\ldots,u_M$, and define 
$$ f_M(v) := (\sum_{m=1}^M f(v-u_m)^n)^{1/n}.$$
Then $f_M$ clearly takes values in $\{0\} \cup [1,+\infty)$, and
\begin{equation}\label{mmm}
 \|f_M\|_{\ell^n(F^n)}^n = \sum_{m=1}^M \| f(\cdot - u_m)\|_{\ell^n(F^n)}^n = M \|f\|_{\ell^n(F^n)}^n
 \end{equation}
and thus (by construction of $M$)
$$ K_0 \|f_M\|_{\ell^n(F^n)} \leq \lambda \ll K_0 \|f_M\|_{\ell^n(F^n)} .$$
On the other hand, since $f_M(v) \geq f(v-u_m)$ for every $1 \leq m \leq M$, we also have
$$ \sum_{v \in \gamma_j(F)+u_m \backslash F^{n-1}} f_M(v) \geq \lambda$$
for all $1 \leq j \leq J$ and $1 \leq m \leq M$.  Applying Proposition \ref{red} with $f$ replaced by $f_M$, $\lambda$ replaced by  $K_0 \|f_M\|_{\ell^n(F^n)}$, and $w_1,\ldots,w_J$ replaced by an enumeration of $\Omega$, we conclude that
$$
|\Omega| \ll \frac{|F|^{n-2}}{(K_0 \|f_M\|_{\ell^n(F^n)})^{n-1}} \sum_{v \in F^n} f_M(v)^n\ll \frac{|F|^{n-2}}{\lambda^{n-1}} \sum_{v \in F^n} f_M(v)^n$$
and hence by \eqref{om}, \eqref{mmm} we have
$$
\min( MJ, |F|^{n-1} ) \ll \frac{|F|^{n-2}}{\lambda^{n-1}} M \|f\|_{\ell^n(F^n)}^n.
$$
If $MJ \leq |F|^{n-1}$ then we obtain \eqref{mong} as required.  If instead $MJ > |F|^{n-1}$, then we have
$$
|F|^{n-1} \ll \frac{|F|^{n-2}}{\lambda^{n-1}} M \|f\|_{\ell^n(F^n)}^n.
$$
Using the definition of $M$, we conclude that
$$ \lambda \gg K_0^n |F|.$$
But this contradicts the hypothesis $\lambda \leq |F|$, if $K_0$ is large enough.
\end{proof}

It remains to prove Proposition \ref{dist} in the case $\lambda = K_0 \|f\|_{\ell^n(F^n)}$ for some large $K_0 = K_0(n,d)$.

We begin with some simple reductions.  Rounding $f$ down to the nearest integer (modifying $\lambda$ and $K_0$ appropriately) we may now assume that $f$ is integer-valued.  Since $\lambda \leq |F|$, we may replace $f$ with $\min(f,|F|)$ without affecting the property \eqref{distal}.  So we have reduced to the case where $f$ takes values in the set $\{0,1,\ldots,|F|\}$.  Our task is to show that
\begin{equation}\label{jb}
J \ll |F|^{n-2} \|f\|_{\ell^n(F^n)}.
\end{equation}

To establish this, we use the polynomial method of Dvir \cite{dvir} (or more precisely the weighted refinement of this method as used in \cite{saraf}, \cite{dvir2}).  
Let $1 \leq D < |F|$ be an integer to be chosen later, and consider the $F$-vector space ${\mathcal P}_D$ of polynomials on $F^n$ of degree at most $d$.  Then
\beq
\dim_F {\mathcal P}_D = \binom{n+D}{D} \asymp d^n.
\eeq
Now let $V_f$ be the subspace of polynomials $P$ in ${\mathcal P}_D$ such that for every $v \in F^n$, $P$ vanishes to order at least $f(v)$ at $v$ (i.e. the multivariate Taylor expansion of $P$ at $v$ contains no nonzero terms of degree less than $f(v)$ in the coordinates $x-v$).  The condition at each $v$ imposes $O(f(v)^n)$ linear conditions on ${\mathcal P}_D$; thus, summing over $v$, we find that
\beq
\dim_F {\mathcal P}_D - \dim_F V_f \ll \sum_{v \in F^n} f(v)^n = \|f\|_{\ell^n(F^n)}^n.
\eeq
In particular, if we choose $D$ equal to a suitably large multiple of $\|f\|_{\ell^n(F^n)}$, the dimension of $V_f$ is positive, so we can choose a nonzero $P$ in $V_f$.

Let $x_n$ be the defining function of $F^{n-1}$ in $F^n$.  We can factor $P = x_n^j Q$ for some $j \geq 0$, where $Q$ is a polynomial which does not contain $x_n$ as a factor.  Of course, $Q$ also has degree at most $d$, and vanishes with degree at least $f(v)$ for any $v \in F^n \backslash F^{n-1}$.

Now let $\gamma_j$ be one of the curves in Proposition \ref{dist}.  We claim that this curve is contained in the algebraic hypersurface $\{Q=0\}$.  We emphasize that we have the geometric, not the combinatorial notion of containment in mind; that is, we are asserting that the restriction of $Q$ to $\gamma_j$ is the zero function on the algebraic curve $\gamma_j$, not merely that it vanishes on every $F$-rational point of $\gamma_j$.

Suppose on the contrary that $\gamma_j$ is not contained in the vanishing locus of $Q$.  Then the algebraic set $\{ v \in \gamma_j: Q(v) = 0 \}$ would have dimension zero, and by Bezout's theorem (Lemma \ref{bezout}) would have degree $O(d) = O( \|f\|_{\ell^n(F^n)} )$.  On the other hand, by construction this set contains $v$ with multiplicity at least $f(v)$ for each $v \in \gamma_j(F) \backslash F^{n-1}$, thus
$$ \sum_{v \in \gamma_j(F) \setminus F^{n-1}} f(v) \ll \|f\|_{\ell^n(F^n)}.$$
But this contradicts \eqref{distal} and the choice of $\lambda$, if $K_0$ is large enough, and the claim follows.

We have shown that the hypersurface $\{Q=0\}$ contains $\gamma_j$ for each $1 \leq j \leq J$, and in particular contains the points $w_1,\ldots,w_J$.  Thus,
$$ |\{ w \in F^{n-1}: Q(w) = 0 \}| \geq J.$$
On the other hand, by construction $Q$ restricts to a non-trivial polynomial on $F^{n-1}$ of degree at most $d$.  By Lemma \ref{size-lem} we have
$$ J \ll |F|^{n-2} D$$
and \eqref{jb} follows by our choice of $d$.  The proof of Proposition \ref{dist} and (thus Theorem \ref{alg-thm2}) is now complete.

\section{The general case}\label{general-sec}

We now prove Theorem \ref{alg-thm}.  As in the statement of the theorem, all implied constants are allowed to depend on $N, d, n$.   The case $n=1$ is trivial (note that $|W(F)| = O(1)$ in this case), so we shall assume $n \geq 2$.  We may also assume that $|F|$ is large compared with $N, n, d$, as the claim is trivial otherwise.

We begin with the simple observation that the projective case of Theorem \ref{alg-thm} follows from the affine case by a random covering argument, as follows.  Let $T: \P^N \to \P^N$ be a random projective transformation (with coefficients in $F$).  Applying $T$ to all the quantities in Theorem \ref{alg-thm}, restricting from projective space $\P^N$ to affine space $\A^N$ and then applying the affine case of that theorem, we see that
$$
\| \sup_{T\gamma \ni Tw} \sum_{v \in T\gamma(F) \cap F^N} |f \circ T^{-1}(v)| \|_{\ell^n(T W(F) \cap F^N)} \ll |F|^{(n-1)/n} \|f \circ T^{-1}\|_{\ell^n(F^N)},$$
which we rearrange as
$$
\| \sup_{\gamma \ni w} \sum_{v \in \gamma(F) \cap T^{-1} F^N} |f(v)| \|_{\ell^n(W(F) \cap T^{-1} F^N)}^n \ll |F|^{n-1} \|f \|_{\ell^n(T^{-1} F^N)}^n.$$
Taking expectations over all $T$, we conclude that the projective case of Theorem \ref{alg-thm} does indeed follow from the affine case as claimed.

It remains to establish the case when $W \subset \A^N$ is an affine variety. For each irreducible curve $\gamma$ in $\A^N$ of degree $d$, we can write

$$ \sum_{v \in \gamma(F)} f(v) = \sum_{v \in \gamma(F) \backslash W(F)} f(v) + \sum_{v \in \gamma(F) \cap W(F)} f(v).$$
To control the latter sum, observe from Bezout's theorem (Lemma \ref{bezout}) that $\gamma(F) \cap W(F)$ has cardinality $O(1)$ (note that we assume $\gamma$ does not lie in $W$), and thus
$$ \sum_{v \in \gamma(F) \cap W(F)} f(v) \ll \|f\|_{\ell^n(F^N)}.$$
Applying Lemma \ref{size-lem}, we have $|W(F)| \ll |F|^{n-1}$ and so we see that the contribution of this term to is negligible.  

It remains to show that
$$
\| \sup_{\gamma \ni w} \sum_{v \in \gamma(F) \backslash W(F)} |f(v)| \|_{\ell^n(W(F))} \leq C_{N,d,n} |F|^{(n-1)/n} \|f\|_{\ell^n(F^N)}.$$
Clearly we may now drop the assumption that $\gamma$ is not contained in $W$.

By Lemma \ref{subdef} we may write
$$ W = \{ v \in \A^N: P_1(v) = \ldots = P_K(v) = 0\}$$
for $O(1)$ polynomials $P_1,\ldots,P_K$ on $\A^N$ of degree $O(1)$, such that each locus $\{P_k=0\}$ is a hypersurface containing $W$.  It will suffice to show for each $k$ that
$$
\| \sup_{\gamma \ni w} \sum_{v \in \gamma(F) \backslash \{P_k=0\}} |f(v)| \|_{\ell^n(W(F))} \ll |F|^{(n-1)/n} \|f\|_{\ell^n(F^N)}$$
since the desired claim then follows from the triangle inequality.

Next, we can reduce to the case when $P_k$ is the vertical coordinate function $x_N$.  Indeed, given any $P_k$, we can realize $W, \gamma$ as subvarieties of $\A^{N+1}$ by composing the given inclusion in $\A^N$ with the ``graphing'' map $\A^N \ra \A^{N+1}$ defined by  $x \mapsto (x,P_k(x))$.  Since $P_k$ has degree $O(1)$, it is clear that the images of $W, \gamma$ in $\A^{N+1}$ are again varieties of degree $O(1).$ Incrementing $N$ by $1$,  we may now assume that $W \subset \A^{N-1}$.

It thus suffices to show that
\begin{equation}\label{gif}
\| g \|_{\ell^n(W(F))} \ll |F|^{(n-1)/n} \|f\|_{\ell^n(F^N)}
\end{equation}
where $W$ is contained in $\A^{N-1}$, and $g: W(F) \to \R^+$ is the function
\begin{equation}\label{gdef}
g(w) :=  \sup_{\gamma \ni w} \sum_{v \in \gamma(F) \backslash F^{N-1}} |f(v)|.
\end{equation}


We now apply another random projection trick; the idea is to ``flatten'' the $n-1$-dimensional variety $W$ by a linear projection $F^{N-1} \ra F^{n-1}$, thus returning us to the situation of the previous section.  The main point is that a sufficiently generic projection from $W(F)$ to $F^{n-1}$ will have, in a sense, bounded fibers on average. 

Let $T: F^{N-1} \to F^{n-1}$ be a random surjective linear map; we extend this map (by abuse of notation) to the linear map $T: F^N \to F^n$ by defining $T(w,v_N) := (Tw, v_N)$ for $w \in F^{N-1}$ and $v_N \in F$.  We introduce the functions $f_T: F^n \to \R^+$ and $g_T: F^{n-1} \to \R^+$ by the formulae
\begin{equation}\label{fat-def}
f_T(x) := (\sum_{v \in F^N: T(v) = x} |f(v)|^n)^{1/n}
\end{equation}
and
\begin{equation}\label{gt-def}
 g_T(y) := \sup_{w \in W(F): T(w) = y}  g(w)
\end{equation}
for all $y \in F^{n-1}$ and $x \in F^n$, with the convention that the supremum over $w$ in \eqref{gt-def} is zero if no $w$ of the required form exist.

\begin{lemma}\label{projpe-lem} For all $T$, we have
\begin{equation}\label{gtft}
\| g_T \|_{\ell^n(F^{n-1})} \ll |F|^{(n-1)/n} \|f_T\|_{\ell^n(F^n)}
\end{equation}
for all $T$.  
\end{lemma}

\begin{proof}
In view of Theorem \ref{alg-thm2}, it suffices to establish the pointwise estimate
$$ g_T(y) \ll f_T^*(y)$$
for all $y \in F^{n-1}$, where the degree $d$ appearing in the definition of $f_T^*$ is $O(1)$.

Fix $y$.  We may of course assume that $g_T$ is non-zero, which implies by \eqref{gt-def} that there exists $w \in W(F)$ with $T(w)=y$ and an irreducible curve $\gamma$ in $\A^N$ passing through $w$ of degree at most $d$, such that
$$ \sum_{v \in \gamma(F) \backslash F^{N-1}} |f(v)| = g_T(y).$$
Since $g_T$ is non-zero, we see that $\gamma(F)$ is not contained in $F^{N-1}$.  Since $\gamma$ contains $w \in F^{N-1}$, we see that the image of $\gamma(F)$ in $F^n$ contains at least two distinct points; in particular, it is not constant.  It follows from Bezout's theorem (Lemma \ref{bezout}) that
$$ |\gamma \cap T^{-1}(y')| \ll 1$$
for any $y' \in F^n$.  From this and \eqref{fat-def} we see that
$$ \sum_{v \in \gamma(F) \backslash F^{N-1}} |f(v)| \ll \sum_{x \in T(\gamma(F)) \backslash F^{n-1}} |f_T(x)|.$$
By Lemma \ref{proj}, $T(\gamma(F))$ is contained in an irreducible algebraic curve in $\A^n$ of degree $O(1)$; this curve passes through $y$, and the claim follows.
\end{proof}

We raise \eqref{gtft} to the $n^{th}$ power and take expectations to conclude
$$
\E \| g_T \|_{\ell^n(F^{n-1})}^n \ll |F|^{n-1} \E \|f_T\|_{\ell^n(F^n)}^n.$$
By construction, $\|f_T\|_{\ell^n(F^n)} = \|f\|_{\ell^n(F^N)}$, and so to show \eqref{gif} we only need to show that
$$ 
\| g \|_{\ell^n(W(F))}^n \ll \E \| g_T \|_{\ell^n(F^{n-1})}^n.$$
By using distributional formulae such as \eqref{distrib} and linearity of expectation, it suffices to show that
$$ | \{ w \in W(F): g(w) \ge \lambda \}| \ll \E | \{ y \in F^{n-1}: g_T(y) \ge \lambda \} |$$
for each $\lambda > 0$.  

Fix $\lambda$.  If we denote the set on the left-hand side as $\Omega$, then the set on the right-hand side is $T(\Omega)$.
It thus suffices to show that $|T(\Omega)| \gg |\Omega|$ with probability $\gg 1$.
From the Cauchy-Schwarz inequality we have
$$ |\{ (w, w') \in \Omega \times \Omega: T(w) = T(w') \}| \gg \frac{|\Omega|^2}{|T(\Omega)|}$$
and so it suffices by Markov's inequality to show that
$$ \E |\{ (w, w') \in \Omega \times \Omega: T(w) = T(w') \}| \ll |\Omega|.$$
The contribution of the diagonal case $w=w'$ is clearly acceptable, so we may impose the condition $w \neq w'$.  But for any $w \neq w'$, the probability that $T(w)=T(w')$ is just the probability that $w-w'$ is in the kernel of $T$, which is $O(|F|^{1-n})$.  So, by linearity of expectation, the contribution of the non-diagonal pairs $(w,w')$ is $O( |\Omega|^2 |F|^{1-n} )$, which is acceptable since $|W(F)| \ll F^{n-1}$.   The proof of Theorem \ref{alg-thm} is complete.

\begin{remark}  One can also apply Dvir's polynomial method directly to the algebraic variety $W$ without performing the flattening trick.  This suffices to establish results such as Corollary \ref{setcor} when the number $J$ of curves is large (e.g. $J \asymp |F|^{n-1}$), but does not seem to give optimal results when $J$ is small.  In case $W=\A^{n-1}$, we made critical use of the large automorphism group of $W$ in Proposition \ref{red} in order to reduce the small $J$ case to the large $J$ case by superimposing many translations of $f$ by random elements of $\Aut(W)$.  When $W$ is a more general variety, its automorphism group is typically trivial, and we do not know of any substitute for the method of Proposition~\ref{red}.

We remark that this is not so different from Lemma~\ref{size-lem}.  In general, the estimation of the number of $F$-rational points on a variety $W$ over $F$ is a very difficult problem, requiring some knowledge of the cohomology of $W$ with its Frobenius action.  But in cases where $W$ has a large automorphism group -- for instance, when $W$ is a homogeneous space, or in particular affine space -- it can be quite easy to compute $W(F)$ exactly.  So if we are content to compute $W(F)$ up to a multiplicative constant, as in Lemma~\ref{size-lem}, we need only flatten $W$ down to affine space and do our combinatorics on the latter, much simpler variety.


\end{remark}

\section{Variants}\label{variants-sec}

We now consider some variants and generalizations of the above results.  We will not always strive for maximal generality here, instead giving a sample of possible directions in which the above methods can be pushed.  As a consequence, some details will be omitted in the discussion.

\subsection{Reducibility}

In Theorem \ref{alg-thm}, $W$ was assumed to be a variety, and thus irreducible.  However, one can clearly generalize to the case when $W$ is the union of a bounded number of varieties of bounded degree, with the irreducible components of $W$ having dimension $n-1$.  

One cannot, of course, relax the requirement that the curve $\gamma$ be irreducible.  If we did so, reducible curves of the form $\gamma \cup \gamma_0$ for each $w \in W(F)$, where $\gamma$ goes through $w$ and $\gamma_0$ is a fixed curve independent of $w$ on which $f$ is large, would easily contradict \eqref{kakeq}.    


\subsection{Restricted Kakeya maximal functions}
As we saw in Remark~\ref{rem:ndg}, the condition that each curve $\gamma$ is not contained in $W$ is an ingredient necessary to obtain the full range of estimates. By examining the argument in Section~\ref{general-sec}, one sees that if this nondegeneracy condition is further strengthened then we can slightly relax the condition that $W$ is an algebraic set while preserving the full range of estimates.
    
\begin{theorem} Let $N, d, n \geq 1$, let $F$ be a finite field, let $V = \P^N$ (resp. $V = \A^N$), let $U  \subset V$ be a projective (resp. affine) variety of any dimension and of degree at most $d$, and let $W$ be any subset of $U(F)$. Then for any $f: V(F) \to \R$ we have
\begin{equation}
\| \sup_{\gamma \ni w} \sum_{v \in \gamma(F)} |f(v)| \|_{\ell^n(W)} \leq C_{N,d,n} \max(|W|,|F|^{n-1})^{\frac{1}{n}} \|f\|_{\ell^n(V(F))}.
\end{equation}
Here, $w$ ranges over $W$, and for each $w \in W$, the range of the supremum is the set of all irreducible projective (resp. affine) algebraic curves $\gamma$ in $V$ of degree at most $d$, which contain $w$ but are not contained in $U$. The constant $C_{N,d,n} > 0$ depends only on $N,d,n$.
\end{theorem}

\subsection{Nikodym sets and maximal functions}\label{niko-sec}

It is well known (see e.g. \cite{tao:boch-rest}) that Kakeya-type estimates imply ``Nikodym-type'' analogues, in which the $n-1$-dimensional variety $W$ is replaced by the $n$-dimensional variety $V$.  For instance, we have the following analogue of Theorem \ref{dvir-thm}, first observed in \cite{li}:

\begin{theorem}[Nikodym set conjecture for $F^n$]\label{nikodym-thm}  Let $n \geq 1$, let $F$ be a finite field, and let $E \subset F^n$ be a subset of $F^n$ such that for every $x \in F^n \backslash E$, there exists a line $\gamma$ passing through $x$ such that $\gamma \backslash \{x\} \subset E$.  Then $|E| \gg |F|^n$.
\end{theorem}

To see the connection with Theorem \ref{dvir-thm}, note that after applying a projective transformation to a set $E$ of the form in Theorem \ref{nikodym-thm}, one gets something very close to a set $E$ of the form in Theorem \ref{dvir-thm}.  The same projective transformation trick allows us to deduce the following maximal function estimate from Theorem \ref{max-thm}:

\begin{theorem}[Nikodym maximal conjecture for $F^n$]\label{max-thm2}  Let $n \geq 1$, let $F$ be a finite field, and let $f: F^n \to \R$ be a function.  Then
$$
\| \sup_{\gamma \ni x} \sum_{x' \in \gamma} |f(x')| \|_{\ell^n(F^n)} \ll |F| \|f\|_{\ell^n(F^n)}
$$
where the implied constant depends on $n$, and $\gamma$ ranges over lines through $x$.
\end{theorem}

The deduction of Theorem \ref{max-thm2} from Theorem \ref{max-thm} was established in the Euclidean case in \cite{tao:boch-rest}; the proof for finite fields is identical and is omitted.  Alternatively, one can deduce Theorem \ref{max-thm2} from Theorem \ref{alg-thm2} (or Theorem \ref{alg-thm}), which implies in particular that
$$
\| \sup_{\gamma \ni x} \sum_{x' \in \gamma \backslash W} |f(x')| \|_{\ell^n(W)}^n \ll |F|^{n-1} \|f\|_{\ell^n(F^n)}^n
$$
for all hyperplanes $W$ in $F^n$.  Averaging this over all $W$ (cf.~ the Calder\'on-Zygmund method of rotations\cite{cz}) yields the claim.  The same argument also yields the following Nikodym-type variant of Theorem \ref{alg-thm2}:

\begin{theorem}[Nikodym maximal conjecture for curves in $F^n$]\label{alg-thm2-nik}  Let $d, n \geq 1$.  For any $f: F^n \to \R$, we have
$$
\| \sup_{\gamma \ni x} \sum_{x' \in \gamma(F)} |f(x')| \|_{\ell^n(F^n)} \ll |F| \|f\|_{\ell^n(F^n)},
$$
where the supremum is over all irreducible algebraic curves $\gamma$ in $F^n$ of degree at most $d$ that pass through $x$, and the implied constant depends on both $n$ and $d$.
\end{theorem}

One can also establish this theorem directly by repeating the arguments in Section \ref{flat-sec} with some minor changes (for instance, it is no longer necessary to factor out all factors of $x_n$ from the polynomial $P$); we leave the details as an exercise for the interested reader.

The following is a Nikodym-type variant of Theorem \ref{alg-thm}. Since, in the Nikodym problem, there is no longer a natural non-degeneracy condition for the curves $\gamma$, the requirement that $W$ is an algebraic set may be dropped entirely.

\begin{theorem}[Nikodym maximal conjecture for subsets of $\P^N(F)$]\label{alg-thm-niko}  Let $N, d, n \geq 1$, let $F$ be a finite field, and let $W$ be any subset of $\P^N(F)$.  Then for any $f: \P^N(F) \to \R$ we have
$$
\| \sup_{\gamma \ni v} \sum_{v' \in \gamma(F)} |f(v')| \|_{\ell^n(W)} \ll \max(|W|,|F|^n)^{1/n} \|f\|_{\ell^n(\P^N(F))},$$
where for each $v \in W$, the supremum is over all irreducible projective (resp. affine) algebraic curves $\gamma$ of degree at most $d$ in $\P^N(F)$ that pass through $v$, and the implied constants depend on $N, d, n$.
\end{theorem}

To deduce this from Theorem \ref{alg-thm2-nik}, it is necessary to slightly modify the arguments in Section~\ref{general-sec}. First, the random surjective linear map $T:F^{N-1} \rightarrow F^{n-1}$ is replaced by a random surjective linear map $T:F^N \rightarrow F^n.$ We are then no longer able to guarantee, in Lemma~\ref{projpe-lem}, that the image under $T$ of each relevant curve $\gamma(F)$ contains at least two points. To remedy this, we 
define a function $h_T(y)$ identically to $g(y)$ except excluding in each $\sup$ the curves $\gamma$ mapped to only one point by $T$. Redefining $g_T(y) = \sup_{w \in W : T(w) = y} h_T(w)$, we 
prove the lemma with the new $g_T.$ Setting $\Omega_T = \{y \in \Omega : h_T(y) = g(y)\}$ one observes that with probability $\geq .6$ we have $|\Omega_T| \gg |\Omega|$ (here we must delete from $\Omega$ the points corresponding to curves with $|\gamma(F)| = 1,$ but these have negligible contribution to the final estimate as before). Since $|\{(w,w') \in \Omega_T \times \Omega_T : T(w) = T(w')\}| \ll |\Omega|\max(1,|W|/|F|^n)$ with probability $\geq .6$, we conclude that $|T(\Omega_T)| \gg |\Omega| \min(1,|F|^n/|W|)$ with probability $\geq .2$.

\subsection{Kakeya for blowups}

The methods of this paper also apply to the case where our Kakeya maximal function is defined not on the set of points $W(F)$, but on the set of {\em tangent directions} to points of $W(F)$; so that our supremum is not over all curves passing through a particular point $w$, but rather over all curves tangent at $w$ to a given line through $w$.


For simplicity, take $W \subset V=\P^N$ to be a smooth projective variety of dimension $n-k$, and let $E$ be a rank $k$ subbundle of the projectivized normal bundle $\P N_V W := \P( TV|_W / TW )$ of $W$ in $V$.  We can think of $E$ as a way of specifying, for each point $w$ of $W$, a $k$-dimensional subspace of the $(N-n+k)$-dimensional space of tangent vectors at $w$ which are normal to $W$.  (We also require that these subspaces ``vary algebraically'' with $w$.)

Then the {\em blowup} of $V$ along $W$ is a variety $\tilde V$ endowed with a map $\pi: \tilde V \ra V$ which is an isomorphism away from $\tilde W := \pi^{-1}(W)$.  Moreover, $\tilde{W}$ has dimension $N-1$ and is naturally identified with $\P N_V(W)$.  In particular, $E$ determines a closed subvariety of $\tilde{W}$ of dimension $n-1$.  For the basic facts about blowups used here, see \cite[\S II.7]{hart:ag}.

If $\gamma \subset V$ is an irreducible curve not contained in $W$, the preimage $\pi^{-1}(\gamma)$ is a union of some number of copies of $\tilde{W}$ with an irreducible curve $\tilde \gamma \subset \tilde V$, called the {\em strict transform} of $\gamma$.  If $w \in \gamma \cap W$, then the lifted point $\tilde w \in \tilde \gamma \cap \tilde W$ is essentially the tangent vector $\gamma'(w) \in T_w V$ of $\gamma$ at $w$, projected onto $\P N_V(W) \equiv \tilde W$ by the obvious map from $TV|_W$  to $\P N_V(W)$.

The variety $\tilde{V}$ can be embedded in some large projective space $\P^M$ in such a way that the strict transform of a curve of degree $O(1)$ again has degree $O(1)$.  (It suffices to observe that there is a very ample line bundle on $\tilde{V}$ of degree $O(1)$; this follows, for instance, from the explicit construction of a very ample line bundle in \cite[Prop II.7.10]{hart:ag}.)



If $e$ is a point of $E$, we say that a curve $\gamma$ passes through $e$ if its strict transform in $\tilde{V}$ passes through $e$, interpreted as a point of $\P^M$.  Without reference to the blowup, we can think of $e$ as a tangent direction to some point $w$ of $W$, and to say $\gamma$ passes through $e$ is to say it passes through $w$ along the given tangent direction.

Applying Theorem \ref{alg-thm} to the pair $(\P^M, E)$, we obtain the following high codimension variant.

\begin{theorem}[Kakeya maximal conjecture for algebraic varieties over $F$, blowup version]\label{alg-thm-high}  Let $N, d, n, k \geq 1$, let $F$ be a finite field, let $W \subset \P^N$ be a smooth projective variety of dimension $n-k$ and degree at most $d$, and let $E$ be a rank $k$ subbundle of the normal bundle of $W$ in $\P^N$.  Then for any $f: V \to \R$ we have
$$
\| \sup_{\gamma // e} \sum_{v \in \gamma(F)} |f(v)| \|_{\ell^n(E(F))} \ll |F|^{(n-1)/n} \|f\|_{\ell^n(V(F))},
$$
where for each $e \in E(F)$ with base point $w \in W(F)$, the supremum is over all irreducible projective algebraic curves $\gamma$ in $V$ of degree at most $d$ in $\P^N$ that pass through $w$ with projective tangent vector $e$ but are not contained in $W$, and the implied constants depend only on $N,d,n,k$.
\end{theorem}

There is of course an obvious affine counterpart to this theorem, in which all projective varieties are replaced by affine ones; we omit the details.

\begin{example}   If $\gamma$ is a nondegenerate affine conic
\beq
ax^2 + bxy + cy^2 + dx + ey + g = 0
\eeq
and $\ell$ is a line defined by the vanishing of some linear form $mx + ny  - p$ (which we also denote by $\ell$) we say $\gamma$ has $\ell$ as an asymptote if $\gamma$ can be written
\beq
\ell \ell' + D = 0 
\eeq
for some other linear form $\ell'$ and some constant $D$.  If $\omega$ is a point on the line at infinity -- that is, a direction in $\A^2$ -- then $\gamma$ has an asymptote in direction $\omega$ precisely when $\gamma$ intersects the line at infinity at $\omega$.

Let $f$ be a function $\P^2(F) \ra \R$.   For each conic $\gamma$ we define
\beq
F(\gamma) = \sum_{v \in \gamma(F)} |f(v)|.
\eeq
Then we can define a function $f^*$ on the space of directions in $\A^2$ by setting $f^*(\omega)$ to be the supremum of $F(\gamma)$ as $\gamma$ ranges over the space of conics with an asymptote in direction $\omega$.  Thus, the norm of $f^*$ is controlled by Theorem \ref{alg-thm} applied with $V = \P^2$ and $W$ the line at infinity.

Now let $W$ be the point $[0:1:0]$ on the line at infinity, corresponding to the horizontal direction.  Applying Theorem~\ref{alg-thm} directly yields the trivial inequality
\beq
|f^*(\omega)| < C \|f\|_{\ell^1}.
\eeq

We can define a more interesting maximal operator using the blowup formalism of this section, as follows.  Let $f^*$ be a function on the space of horizontal lines $\ell$, defined by setting $f^*(\ell)$ to be the supremum of $F(\gamma)$ as $\gamma$ ranges over the space of conics having $\ell$ as asymptote.  The horizontal lines in $\A^2$ are in bijection with the directions through the point $[0:1:0]$ in $\P^2$, apart from the direction of the line at infinity, and the condition that $\gamma$ has $\ell$ as asymptote implies that the projective closure of $\gamma$ passes through $[0:1:0]$ with tangent direction corresponding to $\ell$.  Now applying Theorem~\ref{alg-thm-high} with $V = \P^2$ and $W=[0:1:0]$ yields an upper bound on the norm of $f^*$.
\end{example}

\subsection{Higher dimensional averages}

\medskip

In the preceding discussion, all averages were over one-dimensional curves $\gamma$.  It is natural to ask what happens if we instead average over $k$-dimensional subvarieties $\pi$ for some $k>1$.  To simplify the discussion we focus on one concrete case, that of the \emph{$k$-plane maximal operator} $T_{n,k}$ for fixed $1 \leq k \leq n$, which maps functions $f: F^n \to \R$ to functions $T_{n,k}: \operatorname{Gr}(F^n,k) \to \R^+$ defined on the Grassmannian of $k$-dimensional subspaces $\pi$ of $F^n$, and is defined by the formula
$$ T_{n,k} f(\pi) := \sup_{a \in F^n} \sum_{v \in \pi+a} |f(v)|.$$

\begin{conjecture}[$k$-plane maximal operator estimate in finite fields]\label{kpconj}  For any $n,k \geq 1$ and $f: F^n \to \R$, we have
$$ \| T_{n,k} f \|_{\ell^q(\Gr(F^n,k))} \ll |\Gr(F^n,k)|^{1/q} \| f \|_{\ell^p(F^n)}$$
whenever $1 \leq p \leq \frac{n}{k}$ and $q \leq (n-k) \frac{p}{p-1}$; the implied constant can depend on $n$ and $k$.
\end{conjecture}

\begin{remark}
Note that the $k=1$ case of this inequality is just \eqref{shoop}.  The various exponents here are sharp, as can be seen by testing $f$ equal to the indicator function of a point, a $k$-plane, or all of $F^n$.  For the best known prior progress towards this result, see \cite{bueti}.
\end{remark}

Although not able to verify the conjecture above, we prove a related estimate where the $\ell^q$ norm is replaced by a certain normalised mixed norm, defined as follows.  For any $1 \leq k \leq n$, any $n$-dimensional vector space $V$ and any exponents $q_1,\ldots,q_k$, we define the norm $\|\cdot\|_{\tilde{\ell}^{q_1, \ldots, q_k}(\Gr(V,k))}$ for $k=1$ by the formula
\beq
\|h\|_{\tilde{\ell}^{q_1}(\Gr(V,1))} = \left(\frac{1}{|\Gr(V,1)|}\sum_{\pi \in \Gr(V,1)} |h(\pi)|^{q_1} \right)^{1/q_1}
\eeq
and then recursively for $k>1$ by the formula
\beq
\|g\|_{\tilde{\ell}^{q_1, \ldots, q_k}(\Gr(V,k))} = \left(\frac{1}{|\Gr(V,1)|}\sum_{\pi \in \Gr(V,1)} \|g(\pi + \cdot)\|_{\tilde{\ell}^{q_2,\ldots,q_{k}}(\Gr(\pi^c,k-1))}^{q_1} \right)^{1/q_1},
\eeq 
where for every subspace of $\pi$ of $V$, we arbitrarily pick a complementary subspace $\pi^c$ (thus $\pi^c+\pi = V$ and $\pi^c \cap \pi = \{0\}$).
If any $q_i = \infty$ above, the corresponding sum is replaced by a $\sup$ in the usual way.

Using an iteration method of Bourgain's \cite{borg:kakeya} and the bound \eqref{exp}, we now obtain

\begin{theorem} \label{mixedq}  For any $n,k \geq 1$ and $f: F^n \to \R$, we have
\beq
 \| T_{n,k} f \|_{\tilde{\ell}^{q_1, \ldots, q_k}(\Gr(F^n,k))} \ll \| f \|_{\ell^\frac{n}{k}(F^n)}
\eeq
with $q_i = \frac{(n-i)(n-i+1)}{n-k};$ the implied constant can depend on $n$ and $k$.
\end{theorem}

Since $\frac{1}{k}(\frac{1}{q_1} + \ldots + \frac{1}{q_k}) = \frac{1}{n}$, one would like to switch the order of the mixed norm and interpolate to deduce Conjecture \ref{kpconj} from the theorem above; this is unfortunately not possible because the exponents $q_i$ are in decreasing order. 

\begin{proof}[Proof of Theorem \ref{mixedq}]
For $k=1$ the bound is identical to \eqref{exp}. Suppose that the bound holds for $k-1$ and all $n$, and that $\pi \in \Gr(F^n,k-1),$ and $\xi \in \Gr(\pi^c,1).$ Note that 
\beq
T_{n,k}f(\xi + \pi) \leq T_{n-k+1,1}\left(\sum_{x \in \pi} |f(x + \cdot)|\right)(\xi).
\eeq
Applying \eqref{exp} to the function $g_{\pi}(y) = \sum_{x \in \pi} |f(x + y)|$ defined on $\pi_c$, we obtain 
\begin{equation} \label{kplanekak}
\|T_{n,k}f\|_{\tilde{\ell}^{r_1, \ldots, r_{k-1},n-k+1}(\Gr(F^n,k))} \ll \|\|g_{\pi}\|_{\ell^{n-k+1}(\pi^c)}\|_{\tilde{\ell}^{r_1, \ldots, r_{k-1}}(\Gr(F^n,{k-1}))}
\end{equation}
for any exponents $r_1, \ldots, r_{k-1}$ where on the right side above, $\pi$ is variable summation in the outer norm. 
Interpolating\footnote{One can for instance use the complex interpolation method here, which works perfectly well in the mixed-norm setting, see e.g. \cite{bergh:interp}.} the hypothesized bound 
\beq
\|\|g_{\pi}\|_{\ell^{\infty}(\pi^c)}\|_{\tilde{\ell}^{\frac{n(n-1)}{n-k+1}, \ldots, n-k+2}(\Gr(F^n,{k-1}))} \ll \|f\|_{\ell^{\frac{n}{k-1}}(F^n)}
\eeq
with the trivial bound
\beq
\|\|g_{\pi}\|_{\ell^{1}(\pi^c)}\|_{\tilde{\ell}^{\infty, \ldots, \infty}(\Gr(F^n,{k-1}))} \ll \|f\|_{\ell^{1}(F^n)}
\eeq
we see that the right side of \eqref{kplanekak} is $\ll \|f\|_{\ell^{\frac{n}{k}}(F^n)}$ provided that each $r_i = \frac{(n-i)(n-i+1)}{n-k}.$ 
\end{proof}

The situation is more satisfactory if we restrict our attention to the analogue of the Kakeya set conjecture.  We say a subset $E \subset F^n$ is a {\em $k$-plane Kakeya set} if it contains a $k$-plane parallel to any given $k$-plane.  (Such sets are also called {\em $(n,k)$-sets.})  Of course, a $k$-plane Kakeya set is also a Kakeya set in the usual sense, so one has $|E| > c_n |F|^n$ by Theorem~\ref{dvir-thm}.  But the $k$-plane Kakeya condition in fact implies a much better lower bound on $|E|$ when $k > 1$.

\begin{proposition}  Let $n,k$ be integers with $2 \leq k < n$.  Let $E$ be a subset of $F^n$ which contains a $k$-plane in every direction, and let $q = |F|$.  Then
\beq
|E| \geq q^n(1-q^{1-k})^{\binom{n}{2}}.
\eeq
for all $q$ sufficiently large relative to $n$.
\label{pr:kplanekakeya}
\end{proposition}

\begin{proof}
We use the method of Dvir together with the ``method of multiplicities'', as in the proof of Proposition~\ref{red}.
Let $m$ be a positive integer and let $\PP_{mq-1}$ be the space of polynomials in $x_1, \ldots, x_n$ of degree at most $mq-1$.  Then
\beq
\dim_F \PP_{mq-1} = {mq +  n - 1 \choose n}
\eeq
We want to find a nonzero polynomial $P$ in $\PP_{mq-1}$ such that $P$ vanishes to order $m$ at each point of $E$.  The vanishing conditions at each point impose ${m + n - 1 \choose n}$ linear conditions on $\PP_{mq-1}$, so we can choose such a $P$ as long as
\begin{equation}
{mq +  n - 1 \choose n} > {m + n - 1 \choose n} |E|.
\label{eq:mq}
\end{equation}
  
Now let $X$ be a $k$-plane contained in $E$.  Then $P$ vanishes to degree $m$ on each point of $X$, and thus must vanish identically on $X$ thanks to the following lemma, a variant of the Schwarz-Zippel lemma.

\begin{lemma}  Let $Q$ be a polynomial on $F^k$ of degree less than $mq$ which vanishes to order $m$ at each point of $F^k$.  Then $Q=0$.
\end{lemma}

\begin{proof}  The statement is clear when $k=1$.  Then proceed by induction; if $k$ is the smallest integer where the lemma doesn't hold, then $Q$ must restrict to $0$ on each of the $(q^{k+1}-q)/(q-1)$ $F$-rational hyperplanes in $F^k$.  Write $\Pi$ for the product of all the $F$-rational linear forms on $F^k$.  Then we can write
\beq
Q = Q_0 \Pi^M
\eeq
for some $M > 0$, and some $Q_0$ not a multiple of $\Pi$.   Now $\Pi$ vanishes at each point of $F^k$ to order $(q^k-1)/(q-1)$, and the degree of $\Pi$ is $(q^{k+1}-q)/(q-1)$.  Thus $Q_0$ vanishes at each point to order $m_0 = m - M(q^k - 1)/(q-1)$ and has degree less than $mq - M(q^{k+1}-q)/(q-1) = m_0q$.  But $Q_0$ is not a multiple of $\Pi$, so $Q_0$ restricts to a nonzero polynomial on some $F$-rational hyperplane, which contradicts the induction hypothesis.
\end{proof}

Let $\bar{P}$ be the leading term of $P$, a homogeneous polynomial of degree at most $mq-1$.  
Since $P$ vanishes identically on a $k$-plane in every direction, $\bar{P}$ must vanish identically on every (projective) $(k-1)$-plane in $\P^{n-1}(F)$.  The following lemma now bounds $\deg P$ from below.

\begin{lemma}  Let $Q$ be a nonzero homogeneous polynomial in $n$ variables which restricts to $0$ on every $F$-rational $(k-1)$-plane on $\P^{n-1}$.  Then $\deg Q \geq (q^{k+1}-1)/(q-1)$.
\end{lemma}

\begin{proof}
Choose a homogeneous form $Q$, and let $k$ be the smallest integer such that $Q$ restricts to $0$ on every $F$-rational $(k-1)$-plane.  Let $X$ be an $F$-rational $k$-plane on which $Q$ doesn't vanish.  Then the restriction of $Q$ to $X$ is a homogeneous form in $k+1$ variables which is a multiple of $(q^{k+1}-1)/(q-1)$ distinct linear forms; thus $Q$ must have degree at least $(q^{k+1}-1)/(q-1)$.  The statement in the lemma follows, since $(q^{k+1}-1)/(q-1)$ increases with $k$.
\end{proof}

To sum up, we have proven that the inequality $\ref{eq:mq}$ can hold only when
\beq
mq - 1 \geq (q^{k+1}-1)/(q-1)
\eeq
We note that the negation of this condition forces $m=1$ when $k=1$; but for larger $k$ it allows substantially more freedom in $m$.  In particular, we can take $m=q^{k-1}$ and conclude that
\beq
|E| \geq {q^k + n - 1 \choose n} {q^{k-1} + n - 1 \choose n}^{-1}
\eeq
and one easily verifies that the latter quantity is bounded below by $q^n(1-q^{1-k})^{n \choose 2}$ once $q$ is large relative to $n$.
\end{proof}

\begin{remark}  Suppose that $1 < k < n$.  Then, by \cite[Theorem
2]{saraf}, $F^{n-k+1}$ contains a Kakeya set of size $2^{-n+k}
q^{n-k+1} + O( q^{n-k} )$.  Embedding this Kakeya set in $F^n$, and
then also adding the set $F^n \backslash F^{n-k+1}$, we obtain a
$k$-plane Kakeya set of size $q^n (1 - (1-2^{-n+k}) q^{1-k} + O(
q^{-k} ) )$.  Comparing this with the upper bound of $q^n (1 -
\binom{n}{2} q^{1-k} + O( q^{-k} ) )$ from Proposition \ref{pr:kplanekakeya}, we
see that the latter is nearly sharp, missing at most  by a constant in the $q^{1-k}$ term.
\end{remark}

\subsection{Kakeya problems over more general rings}

The analogy between Kakeya problems over $\R$ and their finite field analogues is deficient in the very important respect that the latter do not admit multiple scales.  That is, there is no natural notion of ``distance'' in $F^n$ more refined than the trivial distance, in which any two distinct points are separated by distance $1$. 

The problem of multiple scales interferes with any attempt to mimic the proofs of finite field results in the Euclidean setting, and indeed, some facts which hold over finite fields are not true over $\R$.  For instance, the construction of Besicovitch sets depends crucially on the existence of multiple scales, and as a result they have no analogue over finite fields;  Dvir's theorem shows that a subset of $F^n$ containing a line in every direction has measure bounded below by an absolute constant (where measure is normalized so that $F^n$ has measure $1$.)

It is thus an interesting problem to reintroduce multiple scales in an algebraic setting.  One might, for instance, ask about Kakeya problems over the finite Artinian rings $F[x]/x^k$ or $\Z/p^k \Z$. These rings have precisely $k$ scales, rather than the infinite scales available in $\R$; so they should stand as a good ''baby case'' for understanding the role played by the interplay between different scales in Kakeya problems.  Better still would be to use results on the finite rings above to derive results in the limit rings $F[[x]]$ and $\Z_p$, which are closer still to $\R$.

Let $R$ be the ring $F[x]/x^k$, and let $\ic{m} = (x)$ be the maximal ideal of $R$.  By a line in $R^n$ we mean a subset of $R^n$ of the form $a + bR$, with $a \in R^n$ and $b \in \R^n - \ic{m}^n$.  The line determines $b$ up to scalar multiplication by the group of units $R^* = R - \ic{m}$; so by the {\em direction} of a line we mean an element of $(R^n - \ic{m}^n) / R^* = \P^{n-1}(R)$.

Dvir's theorem almost immediately implies a lower bound on the size of Kakeya sets in $R^n$.

\begin{proposition} Let $R$ be the ring $F[x]/x^k$, and let $E \subset R^n$ be a subset containing a line in every direction.  Then
\beq
|E| \geq c^{nk} |R|^n
\eeq
\label{pr:fxxk}
for some $c$ in $(0,1)$.
\end{proposition}
\begin{proof}
Choose an isomorphism $\phi: R^n \ra F^{nk}$ of $F$-vector spaces.  Then $\phi$ carries multiplication by $x$ to an endomorphism $X$ of $F^{nk}$ satisfying $X^k = 0$.

Let $\omega$ be a direction in $\P^{nk-1}(F)$, and let $v$ be a point of $F^{nk}$ lying over $\omega$.  Let $v_0$ be an element of $R^n - \ic{m}^n$ such that $\phi^{-1}(v) \in Rv_0$.  (If $v$ doesn't lie in the image of $X$, we can take $v_0$ to be $\phi^{-1}(v)$ itself.)

By hypothesis, there is some $x \in R^n$ such that $x + Rv_0 \subset E$.  But applying $\phi$, this certainly implies that $\phi(x) + F v \subset \phi(E)$.  So $\phi(E)$ is a Kakeya set, and we conclude (\cite{dvir2},\cite{saraf}) that
\beq
|E| = |\phi(E)| < c^{nk} |F|^{nk}
\eeq
for some $c \in (0,1)$ (which in fact can be taken to be slightly larger than $1/2$ as in \cite{dvir2}).
\end{proof}

Proposition~\ref{pr:fxxk} shows that, when $k$ is held fixed and $F$ varies, the analogue of the Kakeya conjecture holds for $R = F[x]/x^k$.  But when $F$ is fixed and $k$ grows, the situation is different.  The natural analogue of Minkowski dimension for a subset $E$ of $R^n$ is 
\beq
\log |E| / \log |R|.
\eeq
So the bound supplied by Proposition~\ref{pr:fxxk} shows only that the Minkowski dimension is bounded below by $1 + \log c / \log |F|$.  This gives the desired lower bound of $1$ only when $|F|$ is allowed to grow.

This leads us to several natural questions:
\begin{itemize}
\item Can the lower bound in Proposition~\ref{pr:fxxk} be improved to something of the form
\beq
|E| \geq c_n |R|^n?
\eeq
More modestly, is there a bound
\beq
|E| \geq c_{n,\eps} |R|^{n-\eps}
\eeq
for any $\eps > 0$?
\item Let $E$ be a subset of $F[[x]]^n$ containing a line in every direction, and write $E_k$ for the image of $E$ under the projection $F[[x]]^n \ra (F[x]/x^k)^n$.  Are there Besicovitch phenomena in $F[[x]]^n$?  That is, is it possible that
\beq
\lim |E_k| |F|^{-nk} = 0?
\eeq
Of course, such a sequence of $E_n$ would give a negative answer to the first question above.
\item What can be said in the cases $R = \Z/p^k \Z$, or $R = \Z_p$?  The simple argument of Proposition~\ref{pr:fxxk} doesn't work, since $R^n$ is no longer an $F$-vector space.  One might try identifying $Z/p^k\Z$ with $\mathbf{F}_p^k$ via the Witt construction (\cite[II.6]{serr:lf}).
\end{itemize}

\appendix

\section{Review of algebraic geometry}\label{alg}

We now quickly review some basic concepts and results in algebraic geometry which are relevant to this paper.  All the material here is found in standard textbooks, e.g. \cite{hart:ag}.

Throughout this appendix, $F$ is a fixed finite field, and $\overline{F}$ is its algebraic closure.  For any $n \geq 1$, we define \emph{affine space}
$$ \A^n := \{ (x_1,\ldots,x_n): x_1,\ldots,x_n \in \overline{F} \}$$
and \emph{projective space}
$$ \P^n := \{ [x_1,\ldots,x_{n+1}]: (x_1,\ldots,x_{n+1}) \in \A^{n+1} \backslash 0 \}$$
where $[x_1,\ldots,x_{n+1}]$ is the equivalence class of $(x_1,\ldots,x_{n+1})$ modulo dilations.  (More precisely, $\A^n$ and $\P^n$ are {\em schemes} whose $\overline{F}$-points are the sets given above; for the purpose of this paper, the distinction between a scheme and its set of $\overline{F}$-points can be safely ignored by non-experts.)

We can embed $\A^n$ in $\P^n$ by identifying $(x_1,\ldots,x_n)$ with $[x_1,\ldots,x_n,1]$.  We also define the set of $F$-points of $\A^n$ and of $\P^n$ by the formulae
$$ \A^n(F) := F^n = \{ (x_1,\ldots,x_n): x_1,\ldots,x_n \in F \}$$
and
$$ \P^n(F) := {\P}F^n = \{ [x_1,\ldots,x_{n+1}]: (x_1,\ldots,x_{n+1}) \in F^{n+1} \backslash 0 \}.$$

\begin{definition}[Varieties]  A \emph{projective algebraic set} (resp. \emph{affine algebraic set}) is a subset $V$ of a projective space $\P^N$ (resp. affine space $\A^N$) of the form
$$ V = \{ x \in \P^N: P_1(x) = \ldots = P_J(x) = 0 \}$$
where $P_1,\ldots,P_J$ are homogeneous polynomials (resp. polynomials).  A projective (resp. affine) algebraic set is called \emph{irreducible} if it cannot be written as the union of two proper projective (resp. affine) algebraic subsets, in which case we refer to the algebraic set as a \emph{variety}.

The \emph{dimension} of an algebraic set $V$ is the smallest integer $d$ such that intersections of $V$ with generic codimension-$d$ subspaces of $\P^N$ or $\A^N$ are finite.  If $V$ is a projective variety of dimension $d$, then generic codimension $d$ subspaces of $\P^N$ intersect $V$ in a constant number of points, known as the \emph{degree} of $V$.  Every affine variety can be embedded in a unique projective variety of the same dimension (the projective \emph{closure} of the affine variety); we define the degree of the affine variety to be that of its projective closure.

\begin{remark}  The definitions given here are classical in style, and thus subject to certain pathologies; for instance, one might expect the vanishing locus of a degree-$d$ polynomial in $2$ variables to be a curve of degree $d$, and this is generically the case: but if the polynomial is, for instance, $(x+y)^d$, then the vanishing locus is (according to our definition) a curve of degree $1$.  Problems of this kind will change degrees at worst by bounded multiplicative constants, so we will systematically ignore them; in any event, all proofs in the paper work when the better-behaved scheme-theoretic definitions of dimension and degree are used.
\end{remark}

Given a projective (resp. affine) algebraic set $V$, we define its set of  \emph{$F$-points} $V(F)$ to be the set $V(F) = V \cap \P^n(F)$ (resp. $V(F) = V \cap \A^n(F)$).

An \emph{irreducible curve} is an algebraic variety of dimension one.  A \emph{hypersurface} in an algebraic variety $V$ of degree $n$ is a subvariety of $V$ whose dimension is $n-1$.
\end{definition}

\begin{lemma}[Size estimate]\label{size-lem}  Let $V \subset \P^N$ be a projective variety of dimension $n$ and degree $d$.  Then $|V(F)| \leq d(|F|+1)^n$.
\end{lemma}

We note that Lemma~\ref{size-lem} is essentially the same as the Schwarz-Zippel theorem.

\begin{proof}  The statement is clear for $n=0$.  We may thus suppose inductively that $n>0$ and that $V$ is not contained in any hyperplane of $\P^N$.  Then consider a pencil of hyperplanes through some fixed subspace of $\P^N$ of dimension $N-2$, and write $V_0,V_1,\ldots,V_{|F|+1}$ for the intersection of $V$ with each of the hyperplanes in this pencil.  Each of these is a projective variety (or union of varieties) of dimension $n-1$ and degree $d$.  By induction, $|V_j(F)| \leq d (|F|+1)^{n-1}$ for all $j$, and the claim follows.
\end{proof}

\begin{lemma}[Degree bounds complexity]\label{subdef}  Let $V \subset \A^N$ be an affine variety of degree $d$.  Then there exists a set of polynomials $P_1,\ldots,P_J$ in $\A^N$ with $1 \leq J \leq C_{N,d}$, with degree at most $C_{N,d}$, such that
$$ V = \{ v \in \A^N: P_1(v) = \ldots = P_J(v) = 0 \}.$$
Here $C_{N,d}$ is a constant depending only on $N$ and $d$.
\end{lemma}

\begin{proof}
This is Corollary 6.11 of Kleiman's article~\cite{klei:sga6} in SGA6.  (In concrete terms, one may think of Kleiman's theorem as saying that the $n$-dimensional subvarieties of $\A^N$ of degree $d$ fall into finitely many continuous families, on each one of which the invariants like ``minimal number and degree of polynomials needed to cut out $V$'' is controlled.)
\end{proof}





Our main algebro-geometric tool is Bezout's theorem:

\begin{lemma}[Bezout's theorem]\label{bezout} 
If $\gamma \subset \P^N$ is an irreducible algebraic curve of degree $d$, and $V$ is a projective or algebraic set of degree $d'$ that does not contain $\gamma$, then $\gamma \cap V$ is a zero-dimensional algebraic variety of degree at most $dd'$; in particular, the cardinality of $\gamma(F) \cap V(F)$, counting multiplicity, is at most $dd'$.
\end{lemma}

For details, including a precise definition of multiplicity, see \cite[I.7]{hart:ag}.


\begin{lemma}[Projection lemma]\label{proj}  Let $\gamma \subset \A^N$ be an irreducible affine curve of degree $d$, and let $T: \A^N \to \A^n$ be a linear map.  Then $T(\gamma)$ is contained in an irreducible affine curve of degree at most $d$.
\end{lemma}

\begin{proof}  First of all, we can write $\A^N = \A^{N-n} \times \A^n$ in such a way that $T$ is projection onto the second factor.  Then $\A^N$ is an open subvariety of $\P^{N-n} \times \A^n$, and $T$ extends to a projection $\overline{T}$ from $\P^{N-n} \times \A^n$ to $\A^n$.  Write $\overline{\gamma}$ for the closure of $\gamma$ in $\P^{N-n} \times \A^n$; then $\overline{T}(\overline{\gamma})$ is a closed subvariety of $\A^n$, since $\overline{T}$ is a proper morphism.  Now $T(\gamma)$ is an open dense subvariety of $\overline{T}(\gamma)$, so it is the complement of a finite set of points in an affine curve in $\A^n$.  

Let $P$ be a generic hyperplane in $\A^n$; in particular we choose $P$ to avoid the finite set $\overline{T}(\overline{\gamma}) \backslash T(\gamma)$.   Then
\beq
\overline{T}(\overline{\gamma}) \cap P = T(\gamma) \cap P = T(\gamma \cap T^{-1}(P)).
\eeq
Since $P$ doesn't contain $T(\gamma)$, we know that $\gamma$ is not contained in the hyperplane $T^{-1}(P)$; so $\gamma \cap T^{-1}(P)$ is finite of degree at most $d$.  This implies that $\overline{T}(\overline{\gamma})$ has degree at most $d$, as required.

\end{proof}


\begin{thebibliography}{10}

\bibitem{bergh:interp}
J. Bergh, J. L\"ofstr\"om, \emph{Interpolation Spaces: An Introduction},
Springer-Verlag, 1976.

\bibitem{borg:kakeya}
J. Bourgain, \emph{Besicovitch-type maximal operators and
applications to Fourier analysis}, Geom. and Funct. Anal. \textbf{22}
(1991), 147--187.

\bibitem{bueti}
J. Bueti, \emph{An incidence bound for $k$-planes in $F^n$ and a planar variant of the Kakeya maximal function}, preprint.

\bibitem{cz}
A. Calder\'on, A. Zygmund, \emph{On singular integrals}, Amer. J. Math. \textbf{78} (1956), 289--309.
 
\bibitem{dvir} Z. Dvir, \emph{On the size of Kakeya sets in finite fields}, preprint.

\bibitem{dvir2}
Z. Dvir, S. Kopparty, S. Saraf, M. Sudan, \emph{Extensions to the Method of Multiplicities, with applications to Kakeya Sets and Mergers}, preprint.

\bibitem{dvir:dvirwigderson}
Z. Dvir, A. Wigderson.  \emph{Kakeya sets, new mergers, and old extractors.}  In Proceedings of the 2008 49th Annual IEEE Symposium on Foundations of Computer Science.

\bibitem{hart:ag} R. Hartshorne, \emph{Algebraic Geometry}, GTM 52, Springer-Verlag, New York - Heidelberg, 1977.

\bibitem{klei:sga6} S.L.Kleiman, "Les th\'{e}or\`{e}mes du finitude pour le foncteur de Picard," in \emph{Th\'{e}orie des intersections et th\'{e}or\`{e}me de Riemann-Roch} (SGA6), expose XIII, pp.616--666.  LNM 225, Springer-Verlag, Berlin-New York, 1971.
 
\bibitem{li}
L. Li, \emph{On the size of Nikodym sets in finite fields}, preprint.

\bibitem{gerd:kakeya}
G. Mockenhaupt, T. Tao, \emph{Kakeya and restriction phenomena for finite fields}, Duke Math. J. \textbf{121} (2004), 35--74.

\bibitem{pisier}
G. Pisier, \emph{Factorization of operators through $L^{p,\infty}$ or $L^{p,1}$ and noncommutative generalizations}, Math. Ann. \textbf{276} (1986), 105--136.

\bibitem{saraf}
S. Saraf, M. Sudan, \emph{Improved lower bound on the size of Kakeya sets over finite fields}, preprint.

\bibitem{serr:lf}
J.~P. Serre, \emph{Local Fields}, GTM 67, Springer-Verlag, New York - Heidelberg, 1979.
\bibitem{stein}
E. Stein, \emph{Limits of sequences of operators}, Ann. Math. \textbf{74} (1961), 140--170.

\bibitem{tao:boch-rest}
T. Tao, \emph{The Bochner-Riesz conjecture implies the Restriction
conjecture}, Duke Math J. \textbf{96} (1999), 363--376.

\bibitem{wolff:survey} 
T. Wolff, \emph{Recent work connected with the Kakeya problem},  Prospects in mathematics (Princeton, NJ, 1996), 129--162, Amer. Math. Soc., Providence, RI, 1999.

\end{thebibliography}
\end{document}